\documentclass[]{amsart}
\usepackage{amsfonts}

\usepackage{amssymb}
\usepackage{amsmath}
\usepackage{amsthm}
\usepackage{amscd}

\usepackage{color}

\newtheorem{theorem}{Theorem}[section]

\newtheorem{lemma}{Lemma}[section]
\newtheorem{proposition}[theorem]{Proposition}
\newtheorem{rem}{Remark}

\definecolor{darkgreen}{cmyk}{1,0,1,.2}

\renewcommand\paragraph[1]{\medskip\textbf{#1} }

\begin{document}
\title[EISENSTEIN--PICARD MODULAR GROUP]{GENERATORS OF THE EISENSTEIN--PICARD MODULAR GROUP IN THREE COMPLEX DIMENSIONS}

\author{BaoHua Xie} \address{College of Mathematics and Econometrics \\Hunan University\\ Changsha, 410082, China}
\email{xiexbh@gmail.com}

\author{JieYan Wang} \address{College of Mathematics and Econometrics \\Hunan University\\ Changsha, 410082, China}
\email{jywang@hnu.edu.cn}

\author{YuePing Jiang} \address{College of Mathematics and Econometrics \\Hunan University\\ Changsha, 410082, China}
\email{ypjiang731@163.com}


\begin{abstract}
Little is known about the generators system of the higher dimensional Picard
modular groups. In this paper, we prove that the higher dimensional Eisenstein--Picard
modular group $\mathbf{PU}(3,1;\mathbb{Z}[\omega_3])$ in three complex dimensions can be generated by four
given transformations.
\end{abstract}

\date{\today}

\maketitle
\section{Introduction}
As the complex hyperbolic analogue of Bianchi groups
$\mathbf{PSL}(2;{\mathcal O}_{d})$, Picard modular groups are
$\mathbf{PU}(n,1;{\mathcal O}_{d})$ where
${\mathcal O}_{d}$ is the ring of algebraic integers of the imaginary quadratic
extension ${\bf Q}(i\sqrt{d})$ for any positive square free integer $d$ (see \cite{HO}).
The elements of the ring ${\mathcal O}_{d}$ can be described (see \cite{HW}):
$$
{\mathcal O}_{d}=\left\{
\begin{array}{ll}
{\bf Z}[i\sqrt{d}] & \mbox{ if } d\equiv 1,2\ (\mbox{mod }4)\\
{\bf Z}[\frac{1+i\sqrt{d}}{2}] & \mbox{ if } d\equiv 3\ \ (\mbox{mod }4).
\end{array}
\right.
$$
It is well known that the ring ${\mathcal O}_{d}$ is Euclidean for positive square free integer $d$ if and only if $d=1, 2, 3, 7, 11$; see \cite{ST}.
In particular, if $d=3,{\mathcal O}_{d}=\mathbb{Z}[\omega_3]$, where $\omega_3=(-1+i\sqrt{3})/2$.
Picard modular groups are the simplest arithmetic lattices in $\mathbf{PU}(n,1)$.

There are many results on Picard modular groups in two complex dimensions.
In \cite{FP}, Falbel and Parker constructed a remarkable simple fundamental domain
of the Eisenstein--Picard modular group $\mathbf{PU}(2,1;\mathbb{Z}[\omega_3])$.
Applying Poincar\'{e} polyhedra theorem, they showed that $\mathbf{PU}(2,1;\mathbb{Z}[\omega_3])$ admits a presentation with two generators.
Similarly, in \cite{FFP}, they obtained a presentation of the Gauss--Picard modular group $\mathbf{PU}(2,1;{\mathcal O}_{1})$.
Francsics and Lax also independently obtained the generators of the Gauss--Picard modular group
acting on the two-dimensional complex hyperbolic space (see \cite{FL1,FL2,FL3}).

However, constructing explicit fundamental domains in complex hyperbolic space is much more difficult than in real hyperbolic space.
Therefore, it is interesting to look for another method to get generators system of Picard modular groups.
In \cite{FFLP}, Falbel et al. gave a simple algorithm to obtain the generators of the Gauss--Picard modular group.
Wang et al. \cite{wxx} showed that this algorithm can also be extended to the Eisenstein--Picard modular group $\mathbf{PU}(2,1;\mathbb{Z}[\omega_3])$.
But it is still open for other Picard modular groups.
Recently, Zhao \cite{Zhao} obtained the generators of the Euclidean Picard modular groups $\mathbf{PU}(2,1;\mathcal {O}_d)$ for $d=2,7,11$ by using a different method.
We note that very little is known about the geometric and algebraic properties, e.g., explicit fundamental domains, generators, presentations of  the higher dimensional Picard modular groups $\mathbf{PU}(n,1;{\mathcal O}_{d})$.

In the present paper, we prove that the method of \cite{FFLP} can also be applied to the higher dimensional Eisenstein--Picard
modular group $\mathbf{PU}(3,1;\mathbb{Z}[\omega_3])$ and obtain a simple description in terms of generators.
More precisely, we prove that the Eisenstein--Picard modular group $\mathbf{PU}(3,1;\mathbb{Z}[\omega_3])$ can be generated by four
transformations, one Heisenberg translation, two Heisenberg rotations and an involution.

This paper is organized as follows. Section 2 gives a brief introduction to  complex hyperbolic geometry
and the Picard modular group. Section 3 contains the main
result and its proof.

\section{Preliminaries}
In this section, we recall some basic materials in complex hyperbolic geometry and Picard modular group.
The general reference on these topics are \cite{Go,Par}.

Let $\mathbb{C}^{n,1}$ denote the vector space $\mathbb{C}^{n+1}$ equipped with the Hermitian form
$$\langle \mathbf{w},\mathbf{z}\rangle=z_{1}\overline{w_{n+1}}+z_{2}\overline{w_{2}}+\ldots +z_{n}\overline{w_{n}}+z_{n+1}\overline{w_{1}}$$
where   $\mathbf{w}$ and $\mathbf{z}$ are the column vectors in $\mathbb{C}^{n,1}$ with entries $z_{1},z_{2},\ldots,z_{n},z_{n+1}$ and $w_{1},w_{2},\ldots,w_{n},w_{n+1}$ respectively. Equivalently, we may write
$$\langle \mathbf{w},\mathbf{z}\rangle=\mathbf{z}^{*}J\mathbf{w}$$ where $\mathbf{w}^{*}$ denote the Hermitian transpose of $\mathbf{w}$ and
\begin{equation*}
J=\left(\begin{array}{ccc}
0 & 0& 1\\
0 & I_{n-1}& 0\\
1 & 0& 0
\end{array}\right).
\end{equation*}

Consider the following subspaces of $\mathbb{C}^{n,1}$:
$$
V_{-}=\{\mathbf{v}\in \mathbb{C}^{n,1}: \langle \mathbf{v},\mathbf{v}\rangle<0\},
$$
$$
V_{0}=\{\mathbf{v}\in \mathbb{C}^{n,1}-\{0\}: \langle
\mathbf{v},\mathbf{v}\rangle=0\}.
$$

Let $\mathbb{P}: \mathbb{C}^{n,1}-\{0\}\rightarrow \mathbb{C}P^{n}$ be the canonical projection onto complex projective space.
Then the {\it complex hyperbolic $n$-space} is defined to be $\mathbf{H}_{\mathbb{C}}^{n}=\mathbb{P}(V_{-})$.  The boundary of the complex hyperbolic $n$-space $\mathbf{H}_{\mathbb{C}}^{n}$ consists of those points in $\mathbb{P}(V_{0})$ together with
a distinguished point at infinity, which denote $\infty$.
The finite points in the boundary of $\mathbf{H}_{\mathbb{C}}^{n}$ naturally
carry the structure of the generalized Heisenberg group (denoted by $\mathcal {H}_{2n-1}$),
which is defined to $\mathbb{C}^{n-1}\times \mathbb{R}$ with the group law
$$
(\xi,\nu)\cdot(z,u)=(\xi+z,\nu+u+2\Im\langle\langle \xi,z \rangle\rangle)
.$$
Here $\langle\langle \xi,z \rangle\rangle=z^{*}\xi$ is the standard positive defined Hermitian form on $\mathbb{C}^{n-1}$. In particular,
we write $\| \xi\|^{2}=\xi^{*}\xi$.

Motivated by this, we define horospherical coordinates on complex hyperbolic space.
To each point $(\xi,\nu, u)\in \mathcal {H}_{2n-1}\times \mathbb{R}_{+}$,
we associated a point $\psi(\xi,\nu, u)\in V_{-}$.
Similarly, $\infty$ and each point $(\xi,\nu, 0)\in \mathcal {H}_{2n-1}\times \{0\}$ is associated to a point in $V_{0}$ by $\psi$.
The map $\psi$ is given by
\begin{equation*}
\psi(\xi,\nu, u)=\left(\begin{array}{c}
(-|\xi|^{2}-u+i\nu)/2\\
\xi\\
1
\end{array}\right),\
\psi(\infty)=\left(\begin{array}{c}
1\\
0\\
\vdots\\
0
\end{array}\right).
\end{equation*}
We also define the origin $0$ to be the point in $\partial\mathbf{H}_{\mathbb{C}}^{n}$  with horospherical coordinates $(0,0,0)$.
We have
\begin{equation*}
\psi(0)=\left(\begin{array}{c}
0\\
0\\
\vdots\\
1
\end{array}\right).
\end{equation*}

The holomorphic isometry group of $\mathbf{H}_{\mathbb{C}}^{n}$ is the group $\mathbf{PU}(n,1)$ of complex
linear transformations, which preserve the above Hermintian form.
That is, for each element $G\in \mathbf{PU}(n,1)$, $G$ is unitary with respect to  $\langle \cdot,\cdot\rangle.$ The corresponding matrix $G=(g_{jk})^n_{i,j=1}$ satisfies the
following condition
\begin{equation}
G^*JG=J,
\end{equation}
where $G^*$ denote the conjugate transpose of the matrix $G$. Picard modular groups for $\mathcal{O}_d$, denoted by $\mathbf{PU}(n,1;\mathcal{O}_d)$, are the subgroups of  $\mathbf{PU}(n,1)$  with entries
in $\mathcal{O}_d$.

 \begin{rem}
 In this paper, the matrixes corresponding to the generators do not always have determinant 1. In fact, the generators  belong to the group
$\mathbf{U}(3,1; \mathbb{Z}[\omega_3])$. In relation to complex hyperbolic isometries, the relevant group is $\mathbf{PU}(3,1; \mathbb{Z}[\omega_3])=\mathbf{U}(3,1; \mathbb{Z}[\omega_3])/\pm I$. By abuse of notation, we will denote the Eisenstein--Picard modular group in three complex dimensions by $\mathbf{PU}(3,1; \mathbb{Z}[\omega_3])$ or $\mathbf{U}(3,1; \mathbb{Z}[\omega_3])$.
 We thank referee for pointing out this.
\end{rem}

   We now discuss the decomposition of complex hyperbolic isometries. We begin by considering those elements fixing $0$ and $\infty$.

The matrix group $\mathbf{U}(n-1)$ acts by Heisenberg rotation. In horospherical coordinates,
the action of $U\in\mathbf{U}(n-1)$ is given by
$$(\xi,\nu, u)\longmapsto (U\xi,\nu, u).$$
The corresponding matrix in $\mathbf{U}(n,1)$ acting on $\mathbb{C}^{n,1}$ is
\begin{equation*}
M_U\equiv\left(\begin{array}{ccc}
1 & 0& 0\\
0 & U& 0\\
0 & 0& 1
\end{array}\right).
\end{equation*}

The positive real numbers $r\in \mathbb{R}^{+}$ act by Heisenberg dilation. In horospherical coordinates, this acting is given by
$$(\xi,\nu, u)\longmapsto (r\xi,r^{2}\nu, r^{2}u).$$ In $\mathbf{U}(n,1)$ the corresponding matrix is
\begin{equation*}
A_r\equiv\left(\begin{array}{ccc}
r & 0& 0\\
0 & I_{n-1}& 0\\
0 & 0& 1/r
\end{array}\right).
\end{equation*}

The Heisenberg group acts by Heisenberg translation. For $(\tau,t)\in \mathcal {H}_{2n-1}$, this is
$$N_{(\tau,t)}(\xi,\nu)=(\tau+\xi, t+\nu+2\Im\langle\langle \tau, \xi \rangle\rangle).$$
As a matrix $N_{(\tau,t)}$ is given by
\begin{equation*}
N_{(\tau, t)}\equiv\left(\begin{array}{ccc}
1& -\tau^{*} & (-\|\tau\|^{2}+it)/2\\
0 & I_{n-1}& \tau\\
0 & 0& 1
\end{array}\right).
\end{equation*}

Heisenberg translations, rotations and dilations generate the Heisenberg similarity group. This is the full subgroup of $\mathbf{U}(n,1)$ fixing $\infty$.

Finally, there is one more important acting, called an inversion $R$, which interchanges $0$ and $\infty$. In matrix notation this map is
\begin{equation*}
R\equiv\left(\begin{array}{ccc}
0& 0& 1\\
0 & -I_{n-1}& 0\\
1 & 0& 0
\end{array}\right).
\end{equation*}

Let $\Gamma_{\infty}$ be the stabilizer subgroup of $\infty$ in $\mathbf{U}(n,1)$. That is
$$\Gamma_{\infty}\equiv\{g\in \mathbf{U}(n,1):\ g(\infty)=\infty \}.$$

\begin{lemma}
Let $G=(g_{jk})^4_{j,k=1}\in \mathbf{U}(3,1)$. Then
$G\in\Gamma_{\infty}$ if and only if $g_{41}=0$.
\end{lemma}

Using Langlands decomposition, any element
$P\in\Gamma_{\infty}$ can be decomposed as a product of a Heisenberg
translation, dilation, and a rotation:
\begin{equation}\label{DeofP}
P= N_{(\tau,t)} A_{r} M_{U} =\left(
            \begin{array}{ccc}
             r& -\tau^{*}U & (-\|\tau\|^2+it)/2r \\
              0 & U & \tau/r \\
              0 & 0 & 1/r \\
            \end{array}
          \right),
\end{equation}
The parameters satisfy the corresponding conditions. That is, $U\in \mathbf{U}(n-1),r\in \mathbb{R}^{+}$ and $(\tau,t)\in \mathcal {H}_{2n-1}$.

\section{The  main result and its proof}
In this section we extend the method in \cite{FFLP} to the higher dimensional Eisenstein--Picard
modular group $\mathbf{U}(3,1; \mathbb{Z}[\omega_3])$.

Let $\mathbf{U}(2; \mathbb{Z}[\omega_{3}])$ be the unitary group $\mathbf{U}(2)$ over the ring $\mathbb{Z}[\omega_{3}]$. Recall that the
unitary matrix $A\in\mathbf{U}(2)$ is of the following form
\begin{equation*}
\mathbf{U}(2)=\left\{A=\left(\begin{array}{cc}
a & b\\
 -\lambda \overline{b}& \lambda \overline{a}
\end{array}\right):  |\lambda|=1, |a|^{2}+|b|^{2}=1 \right\}.
\end{equation*}
Then we can see that the elements in $\mathbf{U}(2; \mathbb{Z}[\omega_{3}])$ are of the following form
\begin{equation*}
\left(\begin{array}{cc}
a & 0\\
 0& b
\end{array}\right),
\left(\begin{array}{cc}
0 & b\\
 a& 0
\end{array}\right)
\end{equation*} where $a,b=\pm 1,\pm \omega_{3},\pm \omega_{3}^{2}$.

It is  easy to find that
\begin{equation*}
\left\{\left(\begin{array}{cc}
a & 0\\
 0& b
\end{array}\right): a,b=\pm 1,\pm \omega_{3},\pm \omega_{3}^{2}\right\}
\end{equation*}
 can be generated by
\begin{equation*}
\left(\begin{array}{cc}
1 & 0\\
 0& -\omega_{3}
\end{array}\right),
\left(\begin{array}{cc}
-\omega_{3} & 0\\
 0& 1
\end{array}\right).
\end{equation*} We also note that
\begin{equation*}
\left(\begin{array}{cc}
0 & 1\\
 1& 0
\end{array}\right)
\left(\begin{array}{cc}
a & 0\\
 0& b
\end{array}\right)=\left(\begin{array}{cc}
0 & b\\
 a& 0
\end{array}\right),
\end{equation*} and
\begin{equation*}
\left(\begin{array}{cc}
0 & 1\\
 1& 0
\end{array}\right)
\left(\begin{array}{cc}
-\omega_3 &0\\
0& 1
\end{array}\right)\left(\begin{array}{cc}
0 & 1\\
 1& 0
\end{array}\right)
=\left(\begin{array}{cc}
1& 0\\
 0& -\omega_3
\end{array}\right).
\end{equation*}

Therefore we have the following result.
\begin{lemma}
$\mathbf{U}(2; \mathbb{Z}[\omega_{3}])$ can be generated by the following two unitary matrixes
\begin{equation*}
U_{1}=\left(\begin{array}{cc}
0 & 1\\
 1& 0
\end{array}\right),
 U_{2}=\left(\begin{array}{cc}
-\omega_3 &0\\
0& 1
\end{array}\right).
\end{equation*}
\end{lemma}

Next, we consider the subgroup
$\Gamma_{\infty}$ of the Picard modular group
$\mathbf{U}(3,1;\mathbb{Z}[\omega_3])$.

\begin{lemma}Let
$\Gamma_{\infty}(3,1;\mathbb{Z}[\omega_3])$ denote the subgroup
$\Gamma_{\infty}$ of Picard modular group
$\mathbf{U}(3,1;\mathbb{Z}[\omega_3])$. Then any element
$P\in\Gamma_{\infty}(3,1;\mathbb{Z}[\omega_3])$ if and only if the
parameters in the Langlands decomposition of $P$ satisfy the
conditions
$$
r=1, t\in\sqrt{3}\mathbb{Z}, \tau=(\tau_{1},\tau_{2})^{T}\in\mathbb{Z}[\omega_3]^{2},
U \in \mathbf{U}(2; \mathbb{Z}[\omega_{3}])$$
the integers $t/\sqrt{3}$ and $\|\tau\|^2$ have the same
parity.
\end{lemma}
\begin{proof}
Let $P\in\Gamma_{\infty}(3,1;\mathbb{Z}[\omega_3])$ be the Langlands decomposition form (2). Then it is easy to see that $r=1$, $t\in\sqrt{3}\mathbb{Z}$ and $U \in \mathbf{U}(2; \mathbb{Z}[\omega_{3}])$.  Since the entries $\tau_{1},\tau_{2}$ of $\tau$ and the entry $(-\|\tau\|^2+it)/2$ are in the ring $\mathbb{Z}[\omega_3]$,
we get that  $t/\sqrt{3}\in
\mathbb{Z}$ and $|\tau_{1}|^2+|\tau_{2}|^2\in\mathbb{Z}$. Further more, they have the same parity.
\end{proof}

\begin{proposition} Let $\Gamma_{\infty}(3,1;\mathbb{Z}[\omega_3])$ be stated as above. Then $\Gamma_{\infty}(3,1;\mathbb{Z}[\omega_3])$  is generated
by the Heisenberg translation $N_{1}=N_{((1,0)^T,\sqrt{3})}$ and the Heisenberg rotations $M_{U_{i}}(i=1,2)$.
\end{proposition}

\begin{proof}  Suppose $P\in\Gamma_{\infty}(3,1;\mathbb{Z}[\omega_3])$. According to Lemma 3.2, there is no dilation
component in its Langlands decomposition, that is
$$
P=N_{(\tau,t)}M_U=\left(
  \begin{array}{ccc}
    1 & -{\tau}^* & (-||\tau||^2+it)/2 \\
    0 & I_{2} & \tau \\
    0 & 0 & 1 \\
  \end{array}
\right)\left(
         \begin{array}{ccc}
           1 & 0 & 0 \\
           0 & U& 0 \\
           0 & 0 & 1 \\
         \end{array}
       \right).
$$
Since the unitary matrix $U\in \mathbf{U}(2;\mathbb{Z}[\omega])$. Then the rotation component of $P$ in
the Langlands decomposition is generated by $M_{U_{i}}(i=1,2)$ by Lemma 3.1.

We now consider the Heisenberg translation part of $P$, $N_{(\tau,t)}$.
Let $$\tau=(a_{1}+b_{1}\omega_3, a_{2}+b_{2}\omega_3)^{T},$$ where $a_{1}, b_{1}, a_{2}, b_{2}\in\mathbb{Z}$, since
$\tau\in\mathbb{Z}[\omega_3]^{2}$. Then $N_{(\tau,t)}$ splits as
\begin{equation}
\begin{split}
N_{(\tau,t)}&=N_{((a_{1}+b_{1}\omega_3, a_{1}+b_{1}\omega_3)^T,t)}\\
&=N_{((a_{1},0)^T,\sqrt{3}a_{1})}\circ
N_{((b_{1}\omega_3, 0)^T,\sqrt{3}b_{1})}\circ N_{((0, a_{2})^T,\sqrt{3}a_{2})}\circ N_{((0, b_{2}\omega_3)^T,\sqrt{3}b_{2})}\\
&\quad \circ N_{((0,0)^T,t+\sqrt{3}a_{1}b_{1}-\sqrt{3}a_{1}-\sqrt{3}b_{1}+\sqrt{3}a_{2}b_{2}-\sqrt{3}a_{2}-\sqrt{3}b_{2})}.
\end{split}
\end{equation}
Here the Heisenberg translations
$$N_{((a_{1},0)^T,\sqrt{3}a_{1})},
N_{((b_{1}\omega_3, 0)^T,\sqrt{3}b_{1})}, N_{((0, a_{2})^T,\sqrt{3}a_{2})},N_{((0, b_{2}\omega_3)^T,\sqrt{3}b_{2})}$$
can be written as follows
\begin{equation}
\begin{split}
N_{((a_{1},0)^T,\sqrt{3}a_{1})}&=N^{a_{1}}_{((1,0)^T,\sqrt{3})},\\
N_{((b_{1}\omega_3, 0)^T,\sqrt{3}b_{1})}&=N^{b_{1}}_{((\omega_3, 0)^T,\sqrt{3})},\\
N_{((0, a_{2})^T,\sqrt{3}a_{2})}&=N^{a_{2}}_{((0, 1)^T,\sqrt{3})},\\
N_{((0, b_{2}\omega_3)^T,\sqrt{3}b_{2})}&=N^{b_{2}}_{((0, \omega_3)^T,\sqrt{3})}.
\end{split}
\end{equation}

We claim that the number
$$\left(t-\sqrt{3}(-a_{1}b_{1}+a_{1}+b_{1}-a_{2}b_{2}+a_{2}+b_{2})\right)/2\sqrt{3}$$ is an integer, namely,
$$\frac{t}{\sqrt{3}}-(-a_{1}b_{1}+a_{1}+b_{1}-a_{2}b_{2}+a_{2}+b_{2})\in 2\mathbb{Z}.$$  According to Lemma 3.2,
the integers $t/\sqrt{3}$ and
\begin{equation}
\begin{split}
\|\tau\|^2&=|a_{1}+b_{1}\omega_3|^2+|a_{2}+b_{2}\omega_3|^2\\
&=a_{1}^2-a_{1}b_{1}+b_{1}^2+a_{2}^2-a_{2}b_{2}+b_{2}^2
\end{split}
\end{equation}
 have the same parity. It can be
easily seen that $$a_{1}^2-a_{1}b_{1}+b_{1}^2+a_{2}^2-a_{2}b_{2}+b_{2}^2+(-a_{2}b_{2}+a_{2}+b_{2})+(-a_{1}b_{1}+a_{1}+b_{1})$$ $$=a_{1}(a_{1}+1)+b_{1}(b_{1}+1)+a_{2}(a_{2}+1)+b_{2}(b_{2}+1)-2a_{1}b_{1}-2a_{2}b_{2}\in
2\mathbb{Z}.$$

 Hence $t/\sqrt{3}$ and $-a_{1}b_{1}+a_{1}+b_{1}-a_{2}b_{2}+a_{2}+b_{2}$ have the same
parity. This proves that $$t_{1}=\frac{t-\sqrt{3}(-a_{1}b_{1}+a_{1}+b_{1}-a_{2}b_{2}+a_{2}+b_{2})}{2\sqrt{3}} \in
\mathbb{Z}.$$

Therefore the Heisenberg translation $$N_{((0,0)^T,t+\sqrt{3}a_{1}b_{1}-\sqrt{3}a_{1}-\sqrt{3}b_{1}+\sqrt{3}a_{2}b_{2}-\sqrt{3}a_{2}-\sqrt{3}b_{2})}$$
can be written as
$$N_{((0,0)^T,2\sqrt{3})}^{t_{1}}.$$
The Heisenberg translation $N_{((0,0)^T,2\sqrt{3})}$ can be
generated by $N_{((1,0)^T,\sqrt{3})}$ and $M_{U_{2}}$, that is
$$
N_{((0,0)^T,2\sqrt{3})}=[N_{1}, M_{U_{2}}N_{1}M_{U_{2}}^{-1}].
$$

We also note that $$N_{((0, 1)^T,\sqrt{3})}=M_{U_{1}}N_{1}M_{U_{1}}, N_{((\omega_3, 0)^T,\sqrt{3})}=M^{-2}_{U_{2}}N_{1}M^{2}_{U_{2}}$$    and $$N_{((0, \omega_3)^{T},\sqrt{3})}=M_{U_{1}}N_{((\omega_3, 0)^T,\sqrt{3})}M_{U_{1}}.$$
This proposition is proved.
\end{proof}

Now we prove our main result.
\begin{theorem}\label{thm:label}
The Picard modular group $\mathbf{U}(3,1;\mathbb{Z}[\omega_3])$ is generated
by  Heisenberg translation $N_{1}$,
two Heisenberg rotations $M_{U_{i}}(i=1,2)$ and the involution $R$.
\end{theorem}
\begin{proof} Let $G=(g_{jk})^4_{j,k=1}$ be an
element of the group $\mathbf{U}(3,1;\mathbb{Z}[\omega_3])$. We only need to consider $G\notin \Gamma_{\infty}(3,1;\mathbb{Z}[\omega_3])$.
According to Lemma 2.1, we have $g_{41}\neq 0$. $G$ maps  the point $\infty$ to the point $(g_{11}/g_{41},g_{21}/g_{41},g_{31}/g_{41})$.
Since $G(\infty)$ is in $\partial\mathbf{H}_{\mathbb{C}}^{3}$, then
\begin{equation}2\Re\left(\frac{g_{11}}{g_{41}}\right)=-\left|\frac{g_{21}}{g_{41}}\right|^{2}-\left|\frac{g_{31}}{g_{41}}\right|^{2}\end{equation}
Consider the Heisenberg translation $N_{G(\infty)}$ that maps $(0,0)$ to $G(\infty)$. The translation $N_{G(\infty)}$ is not necessary in
the Picard modular group $\mathbf{U}(3,1;\mathbb{Z}[\omega_3])$. However, we know that
$R N_{G(\infty)}^{-1}G$ belongs to the stabilizer subgroup of $\infty$.
So we will successively approximate $N^{-1}_{G(\infty)}$ by Heisenberg
translations in the Picard modular group to decrease the value
$|g_{41}|^2\in\mathbb{Z}$ until it becomes $0$. Then $G$  can be
expressed as a product of the generators in $\Gamma_{\infty}(3,1;\mathbb{Z}[\omega_3])$ and $R$. Since the ring $\mathbb{Z}[\omega_3]$ is
Euclidean, this approximation process has finitely many steps.

Next we calculate the entry in the lower left corner of the product
\begin{equation}
\begin{split}
G_1&=R N_{(\tau,t)}G\\
&=\left(
\begin{array}{cccc}
 0 & 0 & 0&1 \\
 0 & -1 & 0 &0\\
 0 & 0& -1&0\\
 1 & 0& 0&0
 \end{array}
 \right)\left(\begin{array}{cccc}
 1 & -\overline{\tau_{1}} & -\overline{\tau_{2}}&(-\|\tau\|^{2}+it)/2\\
 0 & 1 & 0 &\tau_{1}\\
 0 & 0& 1&\tau_{2}\\
 0& 0& 0&1
 \end{array} \right)G\\
&= \left(
\begin{array}{cccc}
  0& 0 & 0&1 \\
  0& -1& 0 &-\tau_{1}\\
  0& 0&-1 &-\tau_{2}\\
 1 & -\overline{\tau_{1}}& -\overline{\tau_{2}}&(-\|\tau\|^{2}+it)/2
 \end{array}
 \right)G
 \end{split}
 \end{equation}
 It follows that the entry $g^{(1)}_{41}$ of $G_{1}$ is equal to
 \begin{equation}
\begin{split}
g^{(1)}_{41}&= g_{11}-g_{21}\bar{\tau_{1}}-g_{21}\bar{\tau_{2}}+g_{41}\frac{-\|\tau\|^2+it}{2} \\
 &=g_{41}\left(\frac{g_{11}}{g_{41}}-\frac{g_{21}}{g_{41}}\bar{\tau_{1}}-\frac{g_{31}}{g_{41}}\bar{\tau_{2}}+\frac{-\|\tau\|^2+it}{2}\right)\\
&=-g_{41}\bigg[\left(-\Re\left(\frac{g_{11}}{g_{41}}\right)+\Re\left(\frac{g_{21}}{g_{41}}\bar{\tau_{1}}\right)+
\Re\left(\frac{g_{31}}{g_{41}}\bar{\tau_{2}}\right)+\frac{\|\tau\|^2}{2}\right)\\
&\ \ \ \ -i\left(\Im\left(\frac{g_{11}}{g_{41}}\right)-\Im\left(\frac{g_{21}}{g_{41}}\bar{\tau_{1}}\right)
-\Im\left(\frac{g_{31}}{g_{41}}\bar{\tau_{2}}\right)+\frac{t}{2}\right)\bigg]\\
&=-g_{41}(I_1-iI_2).
\end{split}
 \end{equation}
Using (6), we can simplify $I_1$ to
$$
I_1=\frac{1}{2}\left(\left|\frac{g_{21}}{g_{41}}+\tau_{1}\right|^2+\left|\frac{g_{31}}{g_{41}}+\tau_{2}\right|^2\right).
$$

Let $\frac{g_{21}}{g_{41}}=x_{1}+y_{1}\omega_3$, $\frac{g_{31}}{g_{41}}=x_{2}+y_{2}\omega_3$, $x_{1},y_{1},x_{2},y_{2}\in \mathbb{R}$. Note that
$$\tau=(\tau_{1},\tau_{2})=(a_{1}+b_{1}\omega_3,a_{2}+b_{2}\omega_3)^{T},$$ where $a_{1},b_{1}, a_{2},b_{2}\in \mathbb{Z}$.
In each copy of $\mathbb{C}$, we can simply choose $a_{j}$ and $b_{j}$ so that $a_{j}+b_{j}\omega_3$ is a point in $\mathbb{Z}[\omega_3]$  so that
$\tau_{j}+x_{j}+y_{j}\omega_3$ is as close to the origin as possible. In other words it lies in (Euclidean) Dirichlet domain for the lattice
$\mathbb{Z}[\omega_3]$ centred at the origin. This is a hexagon whose vertices are at
$$\frac{1}{2}+\frac{i}{2\sqrt{3}},\quad \frac{i}{\sqrt{3}},\quad -\frac{1}{2}+\frac{i}{2\sqrt{3}},\quad -\frac{1}{2}-\frac{i}{2\sqrt{3}},\quad -\frac{i}{\sqrt{3}},\quad \frac{1}{2}-\frac{i}{2\sqrt{3}}.$$
All points of this hexagon are at most a distance $1/\sqrt{3}$ from the origin.

Hence, we obtain the upper bound
\begin{equation}
\begin{split}
|I_1|&=\frac{1}{2}\left(\left|x_{1}+y_{1}\omega_3+a_{1}+b_{1}\omega_3\right|^2+\left|x_{2}+y_{2}\omega_3+a_{2}+b_{2}\omega_3\right|^2\right)\\
&\leq
\frac{1}{2}(\frac{1}{3}+\frac{1}{3})\\
&=\frac{1}{3}.
\end{split}
 \end{equation}

The above estimate for $I_{1}$ was suggested by the referee. Since $t/\sqrt{3}\in \mathbb{Z}$ and $|\tau_{1}|^2+|\tau_{2}|^2\in\mathbb{Z}$ have the same parity,  $t/\sqrt{3}$ is an odd or even  number.
In both cases, we can get the inequality
$$|I_2|=\left|\Im\left(\frac{g_{11}}{g_{41}}\right)-\Im\left(\frac{g_{21}}{g_{41}}\bar{\tau_{1}}\right)
-\Im\left(\frac{g_{31}}{g_{41}}\bar{\tau_{2}}\right)+\frac{t}{2}\right|
\leq \frac{\sqrt{3}}{2}$$ by selecting some $t$ in $I_2$.
Therefore, we have the estimation of $g^{(1)}_{41}$
$$
|g^{(1)}_{41}|^2=|g_{41}|^2|I_1+iI_2|^2=|g_{41}|^2(I^2_1+I^2_2)\leq
|g_{41}|^2\left[\left(\frac{1}{3}\right)^2+\left(\frac{\sqrt{3}}{2}\right)^2\right]=\frac{31}{36}|g_{41}|^2.
$$

We can reduce the matrix of
the transformation $G$ to the matrix of a transformation $G_n$ with
$g^{(n)}_{41}=0$ by repeating this approximation procedure finitely
many times. According to Lemma 2.1, this condition implies
that the $G_n$ belongs to the subgroups $\Gamma_{\infty}$. As we
shown in Proposition 3.1, the subgroup $\Gamma_{\infty}$ can be
generated by the Heisenberg translations $N_{1}$
and the Heisenberg rotations $M_{U_{i}}(i=1,2)$.
Since the approximation procedure just uses the transformations in
$\Gamma_{\infty}$ and involution $R$. Hence the proof of
Theorem 3.2 is completed.
\end{proof}
\begin{rem}
It would be interesting to know if this method can be extended to the other higher dimensional Picard modular groups.
\end{rem}
\section*{Acknowledgements}
We would like to thank the referee for his/her careful reading of the manuscript and for making
constructive suggestions.  These have greatly improved the paper.
This work was partially supported by NSF(No.11071059) and
B. Xie also supported by Tianyuan Foundation(No.11126195), NSF(No.11201134) and 'Young teachers support program' of Hunan University.


\begin{thebibliography}{2}

\bibitem{FFP}E. Falbel, G. Francsics, J.R. Parker, \emph{The geometry of Gauss-Picard modular group},  Math. Ann. \textbf{349}(2011), 459--508.

\bibitem{FFLP}E. Falbel, G. Francsics, P.D. Lax, J.R. Parker,
\emph{Generators of a Picard modular group in two complex dimensions},
 Proc. Amer. Math. Soc. \textbf{139}(2011), 2439--2447.

\bibitem{FL1} G. Francsics, P. Lax, \emph{A semi-explicit fundamental
  domain for a Picard modular group in complex hyperbolic space},
   Contemp. Math. \textbf{238}(2005), 211--226.

\bibitem{FL2} G. Francsics, P. Lax, \emph{An explicit fundamental
  domain for a Picard modular group in complex hyperbolic space},
  Preprint (2005).
\bibitem{FL3} G. Francsics, P. Lax, \emph{Analysis of a Picard modular group}, Proc. of National Academy of Sciences USA, \textbf{103}(2006), 11103--11105.

\bibitem{FP}E. Falbel, J. R. Parker, \emph{The geometry of the Eisenstein-Picard
modular group},  Duke Math. J. \textbf{131}(2006), 249--289.

\bibitem{Go} W. M. Goldman, \emph{Complex Hyperbolic Geometry}
(Clarendon, Oxford, 1999).

\bibitem{HO} R. P. Holzapfel, \emph{Invaiants of arithmetric ball quotient sufraces},
 Math. Nachr. \textbf{103}(1981), 117--153.

\bibitem{HW} G. H. Hardy, E. M. Wright, \emph{ An introduction to the theorey
of numbers} (Clarendon, Oxford, 1954).
\bibitem{Par} J. R. Parker, {\it Notes on Complex Hyperbolic
Geometry} (Preprint, 2003).
\bibitem{ST} I.N. Stewart, D.O. Tall, \emph{ Algebraic number theory}, (Chapman and Hall Ltd., 1979).
\bibitem{wxx}J. Wang, Y. Xiao and B. Xie, \emph{Generators of the Eisenstein-Picard
modular groups}, J. Aust. Math. Soc. \textbf{91}(2011), 421--429.

\bibitem{Zhao}T. Zhao, \emph{Generators for the
Euclidean Picard modular groups}, Tran. Amer. Math. Soc. \text{364}(2012), 3241--3263.



\end{thebibliography}
\end{document}